\documentclass[11pt,twoside]{amsart}
\usepackage{mathrsfs}
\usepackage{amsmath}
\usepackage{amsthm}
\usepackage[numbers,sort&compress]{natbib}
\usepackage{amsfonts}
\usepackage{amssymb}
\usepackage{latexsym}
\usepackage[all]{xy}
\usepackage{enumerate}

\date{\empty}
\thanks{}
\allowdisplaybreaks[4] \footskip=15pt
\renewcommand{\uppercasenonmath}[1]{}

\numberwithin{equation}{section} \theoremstyle{plain}
\newtheorem*{thm*}{Main Theorem}
\newtheorem{theorem}{Theorem}[section]
\newtheorem{corollary}[theorem]{Corollary}
\newtheorem*{corollary*}{Corollary}

\newtheorem*{claim*}{Claim}
\newtheorem{lemma}[theorem]{Lemma}
\newtheorem*{lemma*}{Lemma}
\newtheorem{proposition}[theorem]{Proposition}
\newtheorem*{proposition*}{Proposition}
\newtheorem{remark}[theorem]{Remark}
\newtheorem*{remark*}{Remark}
\newtheorem{example}[theorem]{Example}
\newtheorem*{example*}{Example}

\newtheorem*{question*}{Question}
\newtheorem{definition}[theorem]{Definition}
\newtheorem*{definition*}{Definition}

\newtheorem*{acknowledgements*}{ACKNOWLEDGEMENTS}

\usepackage{color}

\begin{document}
\begin{center}
{\large  \bf Characterizations of the $(b, c)$-inverse in a ring}\\
\vspace{0.8cm} {\small \bf Long Wang$^{1,2}$, \ Jianlong Chen$^{1}$
\footnote{Corresponding author: Jianlong Chen.
Email: wangltzu@163.com (Long Wang), jlchen@seu.edu.cn (Jianlong Chen), nieves@fi.upm.es (Nieves Castro-Gonz\'{a}lez).\\
1 The research is supported by the NSFC (11371089), NSF of Jiangsu Province (BK20141327, BK20130599),
Specialized Research Fund for the Doctoral Program of Higher Education (20120092110020).\\
3 The research is supported by Project MTM2010-18057, ``Ministerio de Ciencia e Innovaci\'{o}n" of Spain.}
and Nieves Castro-Gonz\'{a}lez$^{3}$}\\
\vspace{0.6cm} {\rm $^{1}$Department of Mathematics, Southeast University, Nanjing, 210096, China\\
$^{2}$Department of Mathematics, Taizhou University, Taizhou 225300, China\\
$^{3}$Departamento Matem\'{a}tica Aplicada, ETSI Inform\'{a}tica, Universidad Polit\'{e}cnica Madrid, 28660
Madrid, Spain}
\end{center}

\bigskip

{ \bf  Abstract:}  \leftskip0truemm\rightskip0truemm Let $R$ be a ring and $b, c\in R$.
In this paper, we give some characterizations of the $(b,c)$-inverse, in terms of the direct sum decomposition, the annihilator and the invertible elements.
Moreover, elements with equal $(b,c)$-idempotents related to their $(b, c)$-inverses are characterized, and the reverse order rule for the $(b,c)$-inverse is considered.
 \bigskip


\section{Introduction}
Moore-Penrose inverse, Drazin inverse and group inverse, as for the classical generalized inverses, are special types of outer inverses.
In \cite{MP.D1}, Drazin introduced a new class of outer inverse in a semigroup and called it $(b, c)$-inverse.
\begin{definition}  \label{def:bc-inverse}
Let $R$ be an associative ring and let $b, c\in R$. An element $a \in R$ is $(b, c)$-invertible if there exists $y\in R$ such that
\begin{center}
$y \in (bRy) \cap (yRc)$, \quad $yab = b$, \quad $cay = c$.
\end{center}
If such $y$ exists, it is unique and is denoted by $a^{\|(b,c)}$.
\end{definition}
From \cite{MP.D1}, we know that the Moore-Penrose inverse of $a$, with respect to an involution $*$ of $R$, is the $(a^{\ast}, a^{\ast})$-inverse of $a$,
the Drazin inverse of $a$ is the $(a^{j}, a^{j})$-inverse of $a$ for some $j\in \mathbb{N}$, in particular, the group inverse of $a$ is the $(a,a)$-inverse of $a$.

Given two idempotents $e$ and $f$, Drazin introduced the Bott-Duffin $(e, f)$-inverse in \cite{MP.D1}, which can be considered as a particular cases of the $(b, c)$-inverse.
In 2014, Kant\'{u}n-Montiel introduced the image-kernel $(p, q)$-inverse for two idempotents $p$ and  $q$, and pointed out that
an element $a$ is  image-kernel $(p, q)$-invertible if and only if it is Bott-Duffin $(p, 1-q)$-invertible \cite[Proposition 3.4]{GKM}.
In \cite{D.M1}, elements with equal idempotents related to their
image-kernel $(p, q)$-inverses are characterized in terms of
classical invertibility. The topics of research on the image-kernel $(p, q)$-inverse and the Bott-Duffin $(e, f)$-inverse attract wide interest (see \cite{C.X1,C.X2,N2,C.L.Z,D.W,MP.D1,GKM,D.M1}).

This article is motivated by the papers \cite{MP.D1,D.M1}. In \cite{MP.D1}, as a generalization of $(b,c)$-inverse, hybrid $(b,c)$-inverse and annihilator $(b,c)$-inverse were introduced. In section 3,  it is shown that if the $(b,c)$-inverse of $a$ exists, then both $b$ and $c$ are regular.
Further, under the natural hypothesis of both $b$ and $c$ regular, some characterizations of the $(b,c)$-inverse are obtained in terms of the direct sum decomposition, the annihilator and the invertible elements. In particular, we will prove that $(b,c)$-inverse, hybrid $(b,c)$-inverse and annihilator $(b,c)$-inverse are coincident. Some results of the image-kernel $(p, q)$-inverse in \cite{D.M1} are generalized.

If $a$ has a $(b, c)$-inverse, then both $a^{\|(b,c)}a$ and $aa^{\|(b,c)}$ are idempotents. These will be referred as to the $(b,c)$-idempotents associated with $a$.
In \cite{N1}, Castro-Gonz\'{a}lez, Koliha and Wei characterized matrices with the same spectral idempotents corresponding to the Drazin inverses of
these matrices. Koliha and Patr\'{\i}cio \cite{JJK1} extend the results to the ring case.
A similar question for the Moore-Penrose inverse was considered in \cite{P.P1}.  In \cite{D.M1}, Mosi\'{c} gave some characterizations of elements which have the same idempotents related to their image-kernel $(p, q)$-inverses.
It is  of interest to know whether two elements in the ring  have  equal $(b,c)$-idempotents.
In section 4, some characterizations of those elements with equal $(b,c)$-idempotents are given.
Moreover, the reverse order rule for the $(b,c)$-inverse is considered.

\section{Preliminaries}
Let $R$ be an associative ring with unit 1. Let $a\in R$. Recall $a$ is a regular element if there
exists $x \in R$ such that $a = axa$. In this case, the element $x$ is called an inner inverse for
$a$ and we will denote it by $a^{-}$. If the equation $x = xax$ is satisfied, then we say that $a$ is
outer generalized invertible and $x$ is called an outer inverse for $a$. An element $x$ that is both inner and outer inverse of $a$ and commutes with $a$, when it exist, must be unique and is called the group inverse of $a$, denoted by $a^{\#}$.
From now on,  $E(R)$ and $R^{\#}$ stand for the set of all idempotents and the set of all group invertible elements  in $R$.
For the sake of convenience, we introduce some necessary notations.

For an element $a \in R$ and $X\subseteq R$, we define
\begin{center}
$aR := \{ax : x \in R\}$, \quad $Ra := \{xa : x \in R\}$;\\
$l(X) := \{y \in R : \ yx=0 \ \text{for any} \ x\in X\}$, \quad $r(X) := \{y \in R : \ xy=0 \ \text{for any} \ x\in X\}$.
\end{center}
In particular,
\begin{center}
$l(a) := \{y \in R : \ ya=0 \}$, \quad $r(a) := \{y \in R : \ ay=0 \}$,\\
$rl(a)=\{y : xy=0, x\in l(a)\}$ and $lr(a)=\{y : yx=0, x\in r(a)\}$.
\end{center}

Let $p, q \in E(R)$. An element $a\in R$ has an image-kernel $(p, q)$-inverse \cite{GKM,D.M1} if there exists an element $c \in R$ satisfying
\begin{center}
$cac = c,  \quad caR = pR, \quad (1-ac)R = qR.$
\end{center}
The image-kernel $(p, q)$-inverse is unique if it exists, and it will be denoted by $a^{\times}$.
A generalization of the original Bott-Duffin inverse \cite{B.D} was given in \cite{MP.D1}:
let $e, f \in E(R)$, an element $a \in R$ is Bott-Duffin $(e, f)$-invertible if
there exist $y \in R$ such that
$y = ey = yf$, $yae = e$ and $fay = f$. When  $e=f$, the element $y$, if any, is given by $y=e(ae+1-e)^{-1}$ as for the original Bott-Duffin inverse.

The above mentioned generalized inverses are particular cases of the $(b, c)$-inverse where $b$ and $c$ have the property of being both idempotents.
Hence, the research of $(b,c)$-inverse has important significance to the development of the generalized inverse theory.

For the future reference we state two known results.
\begin{lemma}\label{lem1:bc-inverse}
\cite[Theorem 2.2]{MP.D1}
For any given $a, b, c \in R$, there exists the $(b, c)$-inverse $y$ of $a$ if and only if
$Rb=Rt$ and $cR=tR$, where $t=cab$.
\end{lemma}

\begin{lemma}\label{lem2:bc-inverse}
\cite[Proposition 6.1]{MP.D1}
For any given $a, b, c \in R$, $y$ is the $(b, c)$-inverse of $a$ if and only if
$yay = y$,  $yR = bR$ and $Ry = Rc$.
\end{lemma}

\section{Some characterizations of the existence of $(b, c)$-inverse}

Firstly, we will give some lemmas which will be used in the sequel.

\begin{lemma} \label{le:y outer of a}
Let $a, y \in R$ such that  $y$ is an outer inverse of $a$. Then
\begin{enumerate}[(i)]
\item $r(a) \cap yR=\{0\}$.
\item $l(a) \cap Ry=\{0\}$.
\item $Ray=Ry$.
\item $yaR=yR$.
\end{enumerate}
\end{lemma}
\begin{proof} $(i)$. Let $x\in r(a) \cap yR$. Then $ax=0$ and there exists $g\in R$ such that $x=yg$. This gives that $ayg=0$ and, thus, $yayg=yg=0$.
Therefore, $x=0$.\par
$(ii)$. Let $x\in l(a) \cap Ry$. Then $xa=0$ and there exists $h\in R$ such that $x=hy$. It leads to $hya=0$. Then $hyay=hy=0$ and, thus, $x=0$.\par
$(iii)$ and $(iv)$. From $yay=y$ it follows that $yaR=yR$ and $Ry=Ray$.
\end{proof}

\begin{lemma} \label{le:regular}
Let $a \in R$ be regular and $b\in R$. Then
\begin{enumerate}[(i)]
\item  $b$ is regular in case $Ra=Rb$.
\item  $rl(a)=aR$ and $lr(a)=Ra$.
\end{enumerate}
\end{lemma}
\begin{proof}
$(i)$. Since $Ra=Rb$, there exist some $g,h\in R$ such that $a=gb$ and $b=ha$. Hence, using that $a$ is regular, one can see $b=(ha)a^{-}a=ba^{-}a=ba^{-}gb$,
which  means that $b$ is regular.\par

$(ii)$. It is easy to check that $aR\subseteq rl(a)$. Note that $l(a)=l(aa^{-})=R(1-aa^{-})$. For any $x\in rl(a)$, one can get $R(1-aa^{-})x=l(a)x=0$.
This gives $x=aa^{-}x\in aR$ and $rl(a)=aR$.
Similar considerations apply to prove  that $lr(a)=Ra$.
\end{proof}

\begin{proposition} \label{pro:regular}
If $a$ has a  $(b,c)$-inverse, then  $b$, $c$  and $t=cab$  are all of them regular.
\end{proposition}
\begin{proof}
Let $y$ be the $(b,c)$-inverse of $a$. In view of  Definition \ref{def:bc-inverse}, one can see $b=yab\in (bRy)ab\subseteq bRb$. This gives that $b$ is regular. In the same manner one can obtain that $c$ is regular.
Now, on account of  Lemma \ref{lem1:bc-inverse}, we have $Rb=Rt$ and $cR=tR$ since the $(b,c)$-inverse of $a$ exists. From Lemma \ref{le:regular}, we conclude that $t$ is regular.
\end{proof}

In what follows, we will give necessary and sufficient conditions for the existence of the $(b,c)$-inverse when $t=cab$ is regular.

\begin{theorem} \label{th:exists_bc-inverse}
Let $a, b, c \in R$. If $t=cab$ is regular, then the following statements are equivalent:
\begin{enumerate}[(i)]
\item $a$ has a $(b,c)$-inverse.
\item $r(a)\cap bR = \{0\}$ and $R = abR \oplus r(c)$.
\item $r(t)=r(b)$ and $tR=cR$.
\item $l(t)=l(c)$ and $Rt=Rb$.
\item  $l(t)=l(c)$ and $r(t)= r(b)$.
\end{enumerate}
\end{theorem}

\begin{proof}
$(i)\Rightarrow (ii)$ Suppose that  $y$ is the $(b,c)$-inverse of $a$. By Lemma \ref{lem2:bc-inverse}, $yay=y$, $yR = bR$ and $Ry=Rc$.
By Lemma \ref{le:y outer of a} $(i)$, one can see $r(a)\cap yR=\{0\}$, it follows that $r(a)\cap bR=\{0\}$.
Since $ay\in E(R)$, we  have  the decomposition  $R=ayR\oplus r(ay)$.
From  $yR = bR$ we obtain $ayR=abR$. By Lemma \ref{le:y outer of a} $(iii)$ and  $Ry=Rc$, then $Ray=Rc$ and hence  $r(ay)=r(c)$.
Consequently, we have $R=abR\oplus r(c)$.\par

$(ii)\Rightarrow (iii)$. It is clear that $r(b)\subseteq r(t)$. For any $x\in r(t)$, we have $tx=cabx=0$. This means that $abx\in r(c)$. Using that $r(c)\cap abR=\{0\}$ we conclude that $abx=0$. Then $bx \in r(a)\cap bR=\{0\}$. This implies that $bx=0$ and, thus, $x\in r(b)$. Therefore $r(t)= r(b)$.\par
It is clear that $tR\subseteq cR$. Since $R = abR \oplus r(c)$, we can write $1=abg+h$ where  $g\in R$ and $h\in r(c)$. Premultiplaying  by
$c$ gives $c=cabg\in tR$, ensuring that $cR=tR$.\par

$(iii)\Rightarrow (iv)$. Since $tR=cR$,  we have $l(c)=l(t)$. It is clear the  $Rt\subseteq Rb$. Using that $t$ is regular and $r(t)=r(b)$ we obtain that $b(1-t^{-}t)=0$.
 Then $b=bt^{-}t$.  Consequently, $Rt=Rb$.\par

$(iv)\Rightarrow (v)$. It is clear. \par

 $(v)\Rightarrow (i)$.  Since  $r(t)=r(b)$ and $t$ is regular we can prove that $Rt=Rb$ as in the proof of $(iii)\Rightarrow (iv)$. Similarly, from
 $l(t)=l(c)$ and the fact that $t$ is regular we get $tR=cR$. On account of Lemma \ref{lem1:bc-inverse} we conclude that $a$ has a $(b,c)$-inverse.
\end{proof}

In  Theorem \ref{th:exists_bc-inverse}, the implications $(i)\Rightarrow (ii)$ and $(ii)\Rightarrow (iii)$ are valid even if  $t$ is not regular.
 However, we will give a counterexample to show that $(iii)$ does not imply $(iv)$ in general when $t$ is not regular.
\begin{example}
Set $R=\mathbb{Z}$, $a=b=1$ and $c=2$. Clearly, $tR=cR$ and $r(t)=r(b)$, but $Rb\neq Rt$.
\end{example}

When we  replace the hypothesis that $t$ is regular in Theorems \ref{th:exists_bc-inverse} by the condition that both $b$ and $c$ are regular, we obtain the following result.
\begin{theorem} \label{th:exists2_bcinverse}
Let  $a, b, c \in R$. If both $b$ and $c$ are  regular, then the statements (i)-(iv) in Theorem \ref{th:exists_bc-inverse} are equivalent.
\end{theorem}
\begin{proof} We note that in item $(iii)$ condition $tR=cR$ together with $c$ is regular implies that $t$ is regular, in item $(iv)$ $Rt=Rb$ together with $b$ is regular implies that $t$ is regular.
 \end{proof}

\begin{remark}
The statements $(v)\Rightarrow (i)$ in Theorem \ref{th:exists_bc-inverse} is not true, when $b$ and $c$ are regular.
For example, set $R=\mathbb{Z}$, $b=c=1$ and $a=2$. Then $b$ and $c$ are regular.
It is easy to check that $l(t)=l(c)$ and $r(t)=r(b)$, but $t=2$ is not regular. Then $a$ is not $(b,c)$-invertible by Proposition \ref{pro:regular}.
\end{remark}

As a generalization of $(b,c)$-inverse, hybrid $(b,c)$-inverse and annihilator $(b,c)$-inverse  were introduced in \cite{MP.D1}.

\begin{definition}  \label{def:hybrid-bc-inverse}
Let $a, b, c, y \in R$. We say that  $y$ is a hybrid $(b, c)$-inverse of $a$
if
\begin{center} \label{eq:hybrid-bc-inverse}
$yay=y$, \quad  $yR = bR$, \quad   $r(y)=r(c)$.
\end{center}
\end{definition}

\begin{definition}  \label{def:annihilator-bc-inverse}
Let $a, b, c, y \in R$. We say that  $y$ is a annihilator $(b, c)$-inverse of $a$ if
\begin{center} \label{eq:hybrid-bc-inverse}
$yay=y$, \quad  $l(y)= l(b)$, \quad   $r(y)=r(c)$.
\end{center}
\end{definition}

In \cite{MP.D1}, Drazin pointed out that for any given $a, b, c \in R$,
\begin{center}
$(b,c)$-invertible $\Rightarrow$ hybrid $(b, c)$-invertible $\Rightarrow$ annihilator $(b, c)$-invertible.
\end{center}
In what follows, we will prove that the three generalized inverses are coincident whenever $t=cab$ is regular.
\begin{theorem} \label{th:all-inverses}Let  $a, b, c, y\in R$. If $t$ is regular, then the following conditions are equivalent:
\begin{enumerate}[(i)]
\item $y$ is the $(b,c)$-inverse of $a$.
\item $y$ is the  hybrid $(b,c)$-inverse of $a$.
\item  $y$ is the annihilator $(b,c)$-inverse of $a$.
\end{enumerate}
\end{theorem}
\begin{proof}
$(i)\Rightarrow (ii)\Rightarrow (iii)$. These implications are clear.\par
$(iii)\Rightarrow (i)$.
By Definition \ref{def:annihilator-bc-inverse}, we have
$1-ay\in r(y)=r(c)$ and $1-ya\in l(y)=l(b)$. This implies that
$c=cay$ and $b=yab$. Next, we will prove that $r(t)=r(b)$ and $l(t)=l(c)$.
Combining with Theorem \ref{th:exists_bc-inverse} $(v)$, then we can find that
\begin{center}
$a$ is annihilator $(b, c)$-invertible $\Rightarrow$ $a$ is $(b, c)$-invertible.
\end{center}
It is clear that $r(b)\subseteq r(t)$. Let $w\in r(t)$. Then $cabw=0$ and hence $abw\in r(c)=r(y)$.
This implies that $yabw=0$. Then $bw=0$ since $yab=b$. This shows $r(t)\subseteq r(b)$. Therefore, $r(t)=r(b)$. Similarly,  we can prove that $l(c)=l(t)$. Since $a$ has a $(b,c)$-inverse $z$, then $a$ has the annihilator $(b,c)$-inverse  $z$ and by the uniqueness we have $z=y$.
\end{proof}

\begin{theorem} \label{th:ebcinverse}
Let  $a, b, c \in R$. If both $b$ and $c$ are  regular, then the statements (i)-(iii) in Theorem \ref{th:all-inverses} are equivalent.
\end{theorem}
\begin{proof}
We only need to prove that $(iii)\Rightarrow (i)$.
If $y$ is the annihilator $(b,c)$-inverse of $a$, then $l(y)=l(b)$, this gives that $rl(y)=rl(b)$.
Since $b$ and $y$ are regular, we have  $rl(b)=bR$ and $rl(y)=yR$ by Lemma \ref{le:regular} $(ii)$.
This implies that $yR=bR$. Similarly, we can obtain that $Ry=Rc$. Thus, it follows that $y$ is the $(b,c)$-inverse of $a$ by  Lemma \ref{lem2:bc-inverse}.
\end{proof}

The following lemma it is well known.

\begin{lemma} \label{le:idempotent}
Let $a \in R$ and $e \in E(R)$. Then the following conditions are equivalent:
\begin{enumerate}[(i)]
\item  $e \in eaeR \cap Reae$.
\item  $eae + 1 - e$ is invertible (or $ae + 1 - e$ is invertible).
\end{enumerate}
\end{lemma}

\begin{theorem} \label{th:d-bcinverse}
Let  $a, b, c, d \in R$ such that the $(b,c)$-inverse of $a$ exists. Let $e=bb^{-}$ where $b^{-}$ are fixed, but arbitrary inner inverses of $b$.
Then the
following statements are equivalent:
\begin{enumerate}[(i)]
\item $d$ has a $(b,c)$-inverse.
\item $e\in ea^{\|(b,c)}deR\cap Rea^{\|(b,c)}de$.
\item  $a^{\|(b,c)}de+1-e$ is invertible.
\end{enumerate}
In this case,
\begin{equation} \label{eq:d-bcinverse}  d^{\|(b,c)}=(a^{\|(b,c)}de+1-e)^{-1}a^{\|(b,c)}. \end{equation}
\end{theorem}

\begin{proof}
  Firstly, as  $a^{\|(b,c)}$ exists we have $a^{\|(b,c)}\in bR\cap Rc$  by Lemma \ref{lem2:bc-inverse}. Therefore
\begin{equation} \label{eq1:d-bcinverse}
a^{\|(b,c)}=bb^{-}a^{\|(b,c)}=a^{\|(b,c)}c^{-}c.
\end{equation}
From Definition \ref{def:bc-inverse} we  have  $b=a^{\|(b,c)}ab$. Combining with (\ref{eq1:d-bcinverse}), we can write
\begin{equation}\label{eq2:d-bcinverse}
b=ea^{\|(b,c)}c^{-}cab.
\end{equation}
$(i)\Rightarrow(ii)$. Suppose that $d^{\|(b,c)}$ exists. By Definition \ref{def:bc-inverse},  we also have  $c=cdd^{\|(b,c)}$. Substituting this into (\ref{eq2:d-bcinverse}) yields
\begin{equation*}
b=ea^{\|(b,c)}c^{-}(cdd^{\|(b,c)})ab=ea^{\|(b,c)}dd^{\|(b,c)}ab.
\end{equation*}
Multiplying on the right by $b^{-}$ we obtain $e=ea^{\|(b,c)}dd^{\|(b,c)}ae$. Since $d^{\|(b,c)}=ed^{\|(b,c)}$, which follows by  interchanging  $a^{\|(b,c)}$ and $d^{\|(b,c)}$   in (\ref{eq1:d-bcinverse}), we get   $e=ea^{\|(b,c)}ded^{\|(b,c)}ae$. This implies that $e\in ea^{\|(b,c)}deR$.
Similarly,  we can prove that $e\in Rea^{\|(b,c)}de$.\par

$(ii)\Rightarrow(iii)$ See Lemma \ref{le:idempotent}.\par

$(iii)\Rightarrow(i)$ Firstly we note that $ea^{\|(b,c)}=a^{\|(b,c)}$ by (\ref{eq1:d-bcinverse}).  Set $x=ea^{\|(b,c)}de+1-e$. It is clear that $ex=xe$ and $ex^{-1}=x^{-1}e$. Write $y=x^{-1}a^{\|(b,c)}$.
Next, we verify that $y$ is the $(b, c)$-inverse of $d$.

\textbf{Step 1}. $ydy=y$. Indeed,

Using $a^{\|(b,c)}=ea^{\|(b,c)}$, we get
\begin{eqnarray*}
ydy&=&x^{-1}a^{\|(b,c)}dx^{-1}a^{\|(b,c)}=x^{-1}ea^{\|(b,c)}dx^{-1}ea^{\|(b,c)}\\
&=&x^{-1}(ea^{\|(b,c)}de+1-e)ex^{-1}a^{\|(b,c)}\\
&=&x^{-1}ea^{\|(b,c)}=x^{-1}a^{\|(b,c)}=y.
\end{eqnarray*}

\textbf{Step 2}. $bR=yR$.

On account of   $a^{\|(b,c)}=ea^{\|(b,c)}$ and $(1-e)b=0$, one can get
\begin{center}
$b=x^{-1}(ea^{\|(b,c)}de+1-e)b=x^{-1}ea^{\|(b,c)}deb=x^{-1}a^{\|(b,c)}deb=ydeb \in yR$
\end{center}
Meanwhile,  $y=x^{-1}a^{\|(b,c)}=x^{-1}ea^{\|(b,c)}=ex^{-1}a^{\|(b,c)}\in bR$. This guarantees $bR=yR$.

\textbf{Step 3}. $Rc=Ry$.

From Definition \ref{def:bc-inverse}, we  have  $c=caa^{\|(b,c)}$. This leads to
$c=caxx^{-1}a^{\|(b,c)}=caxy\in Ry$. On the other hand, from (\ref{eq1:d-bcinverse}) we conclude that $y=x^{-1}a^{\|(b,c)}=x^{-1}a^{\|(b,c)}c^{-}c\in Rc$. It means that $Rc=Ry$.
\end{proof}

Similarly, we can state the analogue of Theorem \ref{th:d-bcinverse}.

\begin{theorem} \label{th2:d-bcinverse}
Let  $a, b, c, d \in R$ such that the $(b,c)$-inverse of $a$ exists.
Let $f=c^{-}c$ where $c^{-}$ are fixed, but arbitrary inner inverses of $c$.
Then the following statements are equivalent:
\begin{enumerate}[(i)]
\item  $d$ has a $(b,c)$-inverse.
\item $f\in fda^{\|(b,c)}fR\cap Rfda^{\|(b,c)}f$.
\item $fda^{\|(b,c)}+1-f$ is invertible.
\end{enumerate}
In this case,
\begin{equation} \label{eq2:d-bcinverse-2}
d^{\|(b,c)}= a^{\|(b,c)}(fda^{\|(b,c)}+1-f)^{-1}.
\end{equation}
\end{theorem}

\begin{remark} \rm  In case that both $a^{\|(b,c)}$ and $d^{\|(b,c)}$ exist, from Theorem \ref{th:d-bcinverse} and \ref{th2:d-bcinverse},  it may be concluded that
\begin{equation}
\begin{split} \label{eq:d_(b,c)_inverse}
(a^{\|(b,c)}de+1-e)^{-1}=d^{\|(b,c)}ae+1-e;\\
(fda^{\|(b,c)}+1-f)^{-1}=fad^{\|(b,c)}+1-f.
 \end{split}
\end{equation}
 Indeed, since
$d^{\|(b,c)}=(a^{\|(b,c)}de+1-e)^{-1}a^{\|(b,c)}$, we have  $(a^{\|(b,c)}de+1-e)d^{\|(b,c)}=a^{\|(b,c)}.$
Hence,
$$(a^{\|(b,c)}de+1-e)(d^{\|(b,c)}ae+1-e)= a^{\|(b,c)}ae+1-e=1,$$
where the last identity is due to the fact that $a^{\|(b,c)}ae=e$, because $b=a^{\|(b,c)}ab$. Interchanging   the roles of $a$ and $d$  in Theorem \ref{th:d-bcinverse} it follows  that
 $(d^{\|(b,c)}ae+1-e)(a^{\|(b,c)}de+1-e)=1$ and, in consequence, the first identity in (\ref{eq:d_(b,c)_inverse}) holds. The second identity in (\ref{eq:d_(b,c)_inverse}) can be proved in the same manner.
\end{remark}

For any two idempotents $p$ and $q$, we replace $b$ and $c$ by $p$ and $1-q$ respectively in Theorem \ref{th:d-bcinverse} and \ref{th2:d-bcinverse},
we obtain the following corollary.
\begin{corollary}
\cite[Theorem 3.3]{D.M1}
Let $p, q \in  E(R)$ and let $a \in R$ be such that $a^{\times}$ exists. Then for
$d \in R$ the following statements are equivalent:
\begin{enumerate}[(i)]
\item  $d^{\times}$ exists.
\item $1-p+a^{\times}dp$ is invertible.
\item $q+(1-q)da^{\times}$ is invertible.
\end{enumerate}
\end{corollary}

\section{Characterizations of elements with equal $(b,c)$-idempotents}

Let $a^{\|(b,c)}$ exists. Since $a^{\|(b,c)}$ is an outer inverse of $a$, when it exists, then  both $a^{\|(b,c)}a$ and $aa^{\|(b,c)}$ are idempotents. These will be referred  to as the  $(b,c)$-idempotents associated with $a$. We are interested in finding characterizations of those elements in the ring with  equal $(b,c)$-idempotents.

In what follows, we will give necessary and sufficient conditions for $aa^{\|(b,c)} =dd^{\|(b,c)}$. We firstly establish an auxiliary result.

\begin{lemma} \label{le:bc-inverses}
Let $a, b, c, d\in R$ such that $a^{\|(b,c)}$ and $d^{\|(b,c)}$ exist. Let $e=bb^{-}$ and $f=c^{-}c$, where $b^{-}$ and $c^{-}$ are fixed, but arbitrary inner inverses of $b$ and $c$, respectively. Then
\begin{enumerate}[(i)]
\item $d^{\|(b,c)}=d^{\|(b,c)}aa^{\|(b,c)}=a^{\|(b,c)}ad^{\|(b,c)}$.

\item $a^{\|(b,c)}=a^{\|(b,c)}dd^{\|(b,c)}=d^{\|(b,c)}da^{\|(b,c)}$.
\item $e=ed^{\|(b,c)}aa^{\|(b,c)}de=ea^{\|(b,c)}ae=ed^{\|(b,c)}de.$
\item $f=fda^{\|(b,c)}ad^{\|(b,c)}f=fdd^{\|(b,c)}f=faa^{\|(b,c)}f$.
\end{enumerate}
\begin{proof}
$(i)$. In view of  (\ref{eq:d-bcinverse}) and (\ref{eq2:d-bcinverse-2}), with the notation $e=bb^{-}$ and $f=c^{-}c$, we have
\begin{eqnarray*}
d^{\|(b,c)}&=&(a^{\|(b,c)}de+1-e)^{-1}a^{\|(b,c)}=d^{\|(b,c)}aa^{\|(b,c)}\\
&=&a^{\|(b,c)}(fda^{\|(b,c)}+1-f)^{-1}=a^{\|(b,c)}ad^{\|(b,c)}.
\end{eqnarray*}

$(ii)$. We get these equalities  by  interchanging  the roles of $a^{\|(b,c)}$ and $d^{\|(b,c)}$   in previous results.

$(iii)$. By the Definition \ref{def:bc-inverse}, we have  $b=d^{\|(b,c)}db$. Multiplying on the right by $b^{-}$ gives  $e=d^{\|(b,c)}de$.
Similarly, $e=ea^{\|(b,c)}ae$. Multiplying  $(i)$ on the right by $de$ leads to $e=ed^{\|(b,c)}aa^{\|(b,c)}de$.

$(iv)$.  By the definition \ref{def:bc-inverse}, we have  $c=cad^{\|(b,c)}$ and, multiplying on the left by $c^{-}$, we get  $f=fdd^{\|(b,c)}$.
Similarly, $faa^{\|(b,c)}f$. Multiplying  $(ii)$ on the left by $fd$, one can see $f=fd a^{\|(b,c)}ad^{\|(b,c)}f$.
\end{proof}
\end{lemma}

\begin{theorem}
Let $a,b,c,d\in R$ such that $a^{\|(b,c)}$ and $d^{\|(b,c)}$ exist. Then the following statements are equivalent:
\begin{enumerate}[(i)]
\item  $aa^{\|(b,c)} =dd^{\|(b,c)}$.
\item  $aa^{\|(b,c)}dd^{\|(b,c)}=dd^{\|(b,c)}aa^{\|(b,c)}$.
\item $ad^{\|(b,c)}da^{\|(b,c)}=da^{\|(b,c)}ad^{\|(b,c)}$.
\item  $ad^{\|(b,c)} \in R^{\#}$ and $(ad^{\|(b,c)})^{\#} =d a^{\|(b,c)}$.
\item  $da^{\|(b,c)} \in R^{\#}$ and $(da^{\|(b,c)})^{\#} =a d^{\|(b,c)}$.
\end{enumerate}
\end{theorem}

\begin{proof}
$(i)\Leftrightarrow (ii) \Leftrightarrow (iii)$. From Lemma \ref{le:bc-inverses} we obtain
\begin{equation} \label{eq:idempotents}
\begin{split}
aa^{\|(b,c)}=aa^{\|(b,c)}dd^{\|(b,c)}=ad^{\|(b,c)}da^{\|(b,c)};\\
dd^{\|(b,c)}=dd^{\|(b,c)}aa^{\|(b,c)}=da^{\|(b,c)}ad^{\|(b,c)}.
\end{split}
\end{equation}
This leads to
\begin{eqnarray*}
aa^{\|(b,c)}=dd^{\|(b,c)}&\Leftrightarrow& aa^{\|(b,c)}dd^{\|(b,c)}=dd^{\|(b,c)}aa^{\|(b,c)}\\
&\Leftrightarrow& ad^{\|(b,c)}da^{\|(b,c)}=da^{\|(b,c)}ad^{\|(b,c)}.
\end{eqnarray*}

$(iii)\Leftrightarrow (iv)$. Set $x= da^{\|(b,c)}$. We will prove that $x$ is the group inverse of $ad^{\|(b,c)}$.
Combining $(iii)$ with Lemma \ref{le:bc-inverses}, we get
\begin{eqnarray*}
xad^{\|(b,c)}&=&da^{\|(b,c)}ad^{\|(b,c)}=ad^{\|(b,c)}da^{\|(b,c)}=ad^{\|(b,c)}x;\\
ad^{\|(b,c)}x ad^{\|(b,c)}&=& a(d^{\|(b,c)}da^{\|(b,c)})ad^{\|(b,c)}=a(a^{\|(b,c)}ad^{\|(b,c)})=ad^{\|(b,c)}; \nonumber\\
xad^{\|(b,c)}x&=& x  ad^{\|(b,c)}da^{\|(b,c)}= xaa^{\|(b,c)}=da^{\|(b,c)}aa^{\|(b,c)}=x.
\end{eqnarray*}
This implies that $ad^{\|(b,c)} \in R^{\#}$ and $(ad^{\|(b,c)})^{\#} = da^{\|(b,c)}$. Conversely, if the latter holds, then $da^{\|(b,c)}ad^{\|(b,c)}=ad^{\|(b,c)}da^{\|(b,c)}$.\par
$(iii)\Leftrightarrow (v)$. The proof is similar to the previous equivalence.
\end{proof}

We state the result in terms of the other $(b,c)$-idempotent.
\begin{theorem}
Let $a,b,c,d\in R$ such that $a^{\|(b,c)}$ and $d^{\|(b,c)}$ exist. Then the following statements are equivalent:
\begin{enumerate}[(i)]
\item  $a^{\|(b,c)}a =d^{\|(b,c)}d$.
\item $d^{\|(b,c)}da^{\|(b,c)}a=a^{\|(b,c)}ad^{\|(b,c)}d$.
\item  $a^{\|(b,c)}dd^{\|(b,c)}a=d^{\|(b,c)}aa^{\|(b,c)}d$.
\item  $a^{\|(b,c)}d \in R^{\#}$ and $(a^{\|(b,c)}d)^{\#} =d^{\|(b,c)}a$.
\item  $d^{\|(b,c)}a \in R^{\#}$ and $(d^{\|(b,c)}a)^{\#} =a^{\|(b,c)}d$.
\end{enumerate}
\end{theorem}

Next, we consider conditions under which the reverse order rule for the $(b,c)$-inverse of  the product $ad$, $(ad)^{\|(b,c)}=d^{\|(b,c)}a^{\|(b,c)}$ holds.
\begin{theorem} \label{th:ad-bcinverse}
Let $a,b,c,d\in R$ such that $a^{\|(b,c)}$ and $d^{\|(b,c)}$ exist.  Then the following statements are equivalent:
 \begin{enumerate}[(i)]
 \item $ad$ has a $(b,c)$-inverse of the form $(ad)^{\|(b,c)}=d^{\|(b,c)}a^{\|(b,c)}$.
 \item $d^{\|(b,c)}=d^{\|(b,c)}add^{\|(b,c)}a^{\|(b,c)}=d^{\|(b,c)}a^{\|(b,c)}add^{\|(b,c)}$.
 \item $a^{\|(b,c)}=a^{\|(b,c)}add^{\|(b,c)}a^{\|(b,c)}=d^{\|(b,c)}a^{\|(b,c)}ada^{\|(b,c)}$.
 \end{enumerate}
 \end{theorem}
\begin{proof}
$(i)\Leftrightarrow (ii)$. We first assume that  $ad$ has a $(b,c)$-inverse given by $(ad)^{\|(b,c)}=d^{\|(b,c)}a^{\|(b,c)}$.
Then Lemma \ref{le:bc-inverses} is true for $(ad)^{\|(b,c)}$ in place of  $a^{\|(b,c)}$. It follows that
\begin{center}
$d^{\|(b,c)}=d^{\|(b,c)}ad(ad)^{\|(b,c)}=(ad)^{\|(b,c)}add^{\|(b,c)}$.
\end{center}
Substituting $(ad)^{\|(b,c)}=d^{\|(b,c)}a^{\|(b,c)}$ yields
\begin{center}
$d^{\|(b,c)}=d^{\|(b,c)}add^{\|(b,c)}a^{\|(b,c)}=d^{\|(b,c)}a^{\|(b,c)}add^{\|(b,c)}$.
\end{center}
Conversely, if the latter identities hold then $y=d^{\|(b,c)}a^{\|(b,c)}$ is the $(b,c)$-inverse of $ad$.
Indeed, since $d^{\|(b,c)}db=b$ and $c=cdd^{\|(b,c)}$, we have
\begin{eqnarray*}
yady&=&d^{\|(b,c)}a^{\|(b,c)}add^{\|(b,c)}a^{\|(b,c)}=d^{\|(b,c)}a^{\|(b,c)};\\
yadb&=&d^{\|(b,c)}a^{\|(b,c)}adb= d^{\|(b,c)}a^{\|(b,c)}add^{\|(b,c)}db=d^{\|(b,c)}db=b;\\
cady&=&cadd^{\|(b,c)}a^{\|(b,c)}=cdd^{\|(b,c)}add^{\|(b,c)}a^{\|(b,c)}=cdd^{\|(b,c)}=c.
\end{eqnarray*}
$(ii)\Rightarrow (iii)$.  By Lemma \ref{le:bc-inverses} we have
$a^{\|(b,c)}=a^{\|(b,c)}dd^{\|(b,c)}=d^{\|(b,c)}da^{\|(b,c)}$. By $(ii)$, one can see
\begin{center}
$a^{\|(b,c)}=a^{\|(b,c)}d(d^{\|(b,c)}add^{\|(b,c)}a^{\|(b,c)})= (d^{\|(b,c)}a^{\|(b,c)}add^{\|(b,c)})da^{\|(b,c)}$.
\end{center}
Hence, it is easy to get
$a^{\|(b,c)}=a^{\|(b,c)}add^{\|(b,c)}a^{\|(b,c)}= d^{\|(b,c)}a^{\|(b,c)}ada^{\|(b,c)}$.\\
$(iii)\Rightarrow (ii)$. The proof  is similar to $(ii)\Rightarrow (iii)$.
\end{proof}

\begin{theorem}
Let $a,b,c,d\in R$ such that $a^{\|(b,c)}$ and $d^{\|(b,c)}$ exist.   Then the following statements are equivalent:
\begin{enumerate}[(i)]
\item $a^{\|(b,c)}a = dd^{\|(b,c)}$.
\item $a^{\|(b,c)}dd^{\|(b,c)}a=dd^{\|(b,c)}aa^{\|(b,c)}$.
\item $d^{\|(b,c)}da^{\|(b,c)}a=da^{\|(b,c)}ad^{\|(b,c)}$.
\item $a^{\|(b,c)}= dd^{\|(b,c)}a^{\|(b,c)}$ and $d^{\|(b,c)}= d^{\|(b,c)}a^{\|(b,c)}a$.
\item $a^{\|(b,c)}ad^{\|(b,c)} = d^{\|(b,c)}a^{\|(b,c)}a$ and $a^{\|(b,c)}dd^{\|(b,c)}=dd^{\|(b,c)}a^{\|(b,c)}$.
\end{enumerate}
If any of the previous statements is valid, then
$(ad)^{\|(b,c)}=d^{\|(b,c)}a^{\|(b,c)}.$
\end{theorem}
\begin{proof}
$(i) \Leftrightarrow (ii) \Leftrightarrow (iii)$.
From Lemma \ref{le:bc-inverses} we obtain
\begin{equation} \label{eq2:idempotents}
\begin{split} a^{\|(b,c)}a=a^{\|(b,c)}dd^{\|(b,c)}a=d^{\|(b,c)}da^{\|(b,c)}a;\\
dd^{\|(b,c)}=dd^{\|(b,c)}aa^{\|(b,c)}=da^{\|(b,c)}ad^{\|(b,c)}.
\end{split}
\end{equation}
Hence, it gives that
\begin{eqnarray*}
a^{\|(b,c)}a=dd^{\|(b,c)}&\Leftrightarrow& a^{\|(b,c)}dd^{\|(b,c)}a=dd^{\|(b,c)}aa^{\|(b,c)}\\
&\Leftrightarrow& d^{\|(b,c)}da^{\|(b,c)}a=da^{\|(b,c)}ad^{\|(b,c)}.
\end{eqnarray*}

$(i)\Leftrightarrow (iv)$.  The necessary condition is immediate.
Next, we assume that   $a^{\|(b,c)}= dd^{\|(b,c)}a^{\|(b,c)}$ and  $d^{\|(b,c)}= d^{\|(b,c)}a^{\|(b,c)}a$. Then we have
 $a^{\|(b,c)}a= dd^{\|(b,c)}a^{\|(b,c)}a$ and $dd^{\|(b,c)}= dd^{\|(b,c)}a^{\|(b,c)}a$. So $a^{\|(b,c)}a =dd^{\|(b,c)}$, as desired.

$(v)\Leftrightarrow(i)$. The proof is similar to the above.\par

Finally, we will prove that $dd^{\|(b,c)}=a^{\|(b,c)}a$  implies that $(ad)^{\|(b,c)}=d^{\|(b,c)}a^{\|(b,c)}$.
Since $d^{\|(b,c)}=d^{\|(b,c)}a^{\|(b,c)}a$, we have $d^{\|(b,c)}=d^{\|(b,c)}a^{\|(b,c)}add^{\|(b,c)}$.
Moreover, since $d^{\|(b,c)}= d^{\|(b,c)}aa^{\|(b,c)}$ by Lemma \ref{le:bc-inverses}, using $dd^{\|(b,c)}=a^{\|(b,c)}a$, it follows that
\begin{center}
 $d^{\|(b,c)}= d^{\|(b,c)}aa^{\|(b,c)}=d^{\|(b,c)}aa^{\|(b,c)}aa^{\|(b,c)}=d^{\|(b,c)}add^{\|(b,c)}a^{\|(b,c)}$.
\end{center}
By Theorem \ref{th:ad-bcinverse} our  assertion is proved.
\end{proof}

\end{document}